\documentclass[11pt, a4paper]{amsart}
\usepackage{amssymb}
\usepackage{enumerate}
\usepackage{url}

\newtheorem{theorem}{Theorem}

\newtheorem{proposition}[theorem]{Proposition}
\newtheorem{corollary}[theorem]{Corollary}
\newtheorem{lemma}[theorem]{Lemma}

\theoremstyle{remark}

\newtheorem{example}[theorem]{Example}

\newtheorem{remark}[theorem]{Remark}

\title[Uniquely presented monoids]{Uniquely presented finitely generated commutative monoids}
\author{P. A. Garc\'{\i}a-S\'anchez}
\address{Universidad de Granada,
Departamento de \'Algebra}
\email{pedro@ugr.es}
\author{I. Ojeda}
\address{Universidad de Extremadura,
Departamento de Matem\'{a}ticas}
\email{ojedamc@unex.es}

\thanks{The first author is supported by the project MTM2007-62346 and FEDER funds. The second author is partially supported by the project MTM2007-64704, National Plan I+D+I}

\begin{document}

\date{\today}
\maketitle

\begin{abstract}
A finitely generated commutative monoid is uniquely presented if it has a unique minimal presentation. We give necessary and sufficient conditions for finitely generated, combinatorially finite, cancellative, commutative monoids to be uniquely presented. We use the concept of gluing to construct commutative monoids with this property. Finally for some relevant families of numerical semigroups we describe the elements that are uniquely presented.
\end{abstract}

\section*{Introduction}

R\'edei proves in \cite{Redei} that every finitely generated commutative monoid is finitely presented. Since then, its proof has been shortened drastically, and a great development has been made on the study and computation of minimal presentations of monoids, more specifically, of finitely generated subsemigroups of $\mathbb{N}^n,$ known usually as affine semigroups (see for instance \cite{Ros97} and \cite{BCMP} or \cite[Chapter 9]{RGS99} and the references therein). For affine semigroups the concepts of minimal presentations with respect to cardinality or set inclusion coincide, that is to say, any two minimal presentations have the same cardinality (this even occurs in a more general setting, see \cite{RGSU}).

The interest of the study of such kind of monoids and their presentations was partially motivated by their application in Commutative Algebra and Algebraic Geometry (see \cite[Chapter 6]{BH} and \cite{Fulton}).

Recently, new applications of affine semigroups have been found in the so-called Algebraic Statistic. It is precisely in this context, where the problem of deciding under which conditions such monoids have a unique minimal presentation has attracted the interests of a number of researchers. Roughly speaking, convenient algebraic techniques for the study of some statistical models seem to be more interesting for statisticians when certain semigroup associated to the model is uniquely presented (see \cite{Takemura}).

The efforts made to understand the problem of the uniqueness come from an algebraic setting and consists essentially in identifying particular minimal generators in a presentation as $R-$module of the semigroup algebra, where $R$ is a polynomial ring over a field (see \cite{Charalambous07, OjVi2}). So, whole families of uniquely presented monoids have not been determined (with the exception of some previously known cases, see \cite{Oj2}) and techniques for the construction of uniquely presented monoids has not been developed so far.

Here, we propose an approach to the problem of the uniqueness of the minimal presentations from a semigroup theoretic point of view. In a preliminary section, we recall the basic definitions and how minimal presentations of finitely generated, combinatorially finite, cancellative and commutative monoids (which includes affine semigroups) are obtained. Next, in Section \ref{Sect-Betti}, we focus on the elements of the monoid whose factorizations yield these presentations, which we call Betti elements. Section \ref{Sect-UPS} provides a necessary and sufficient condition a monoid must fulfill to be uniquely presented (Corollary \ref{Cor mUP2}). Some results in these sections may be also stated in combinatorial terms by using the simplicial complexes introduced by S. Eliahou in his unpublished PhD thesis (1983), see \cite{Charalambous07} and \cite{OjVi}.

In Section \ref{Sect Glue}, we make extensive use of the gluing of affine semigroups. The concept gluing of semigroups was defined by J.C. Rosales in \cite{Ros97} and was used later by different authors to characterize complete intersection affine semigroup rings. In that section, given a gluing $S$ of two affine semigroups $S_1$ and $S_2,$ we show that $S$ is uniquely presented if and only if $S_1$ and $S_2$ are uniquely presented and some extra natural condition on where $S_1$ and $S_2$ glue holds (Theorem \ref{Th Glue2}). In order to reach this result, we obtain Theorem \ref{Th Glue1} showing that the Betti elements of $S$ are the union of the Betti elements of $S_1,$ $S_2$ and the element in which $S_1$ and $S_2$ glue to produce $S.$ Both theorems may be considered as the main results in this manuscript. Furthermore, Theorem \ref{Th Glue2} may be used to systematically produce uniquely presented monoids as we show in Example \ref{Ejem Glue1}.

Finally, in the last section, we identify all uniquely presented monoids in some classical families of numerical semigroups (submonoids of $\mathbb{N}$ with finite complement in ${\mathbb N}$).

\section{Preliminaries}\label{Sect-pre}

In this section, we summarize some definitions, notations and results that will be useful later in the paper. We refer to the reader to \cite{RGS99} for further information.

Let $S$ denote a \emph{commutative monoid}, that is to say, a set with a binary operation that is associative, commutative and has an identity element which we will denote by $\mathbf{0}.$ Since $S$ is commutative, we will use additive notation. Assume that $S$ is \emph{cancellative} ($\mathbf{a}+\mathbf{b}=\mathbf{a}+\mathbf{c}$ in $S$ implies $\mathbf{b}=\mathbf{c}$). The monoids under study in this paper are also \emph{free of units} ($S\cap (-S)=\{\mathbf{0}\}$). Some authors call these monoids \emph{reduced} (see for instance \cite{RGS99}), others refer to this property as \emph{positivity} (\cite[Chapter 6]{BH}). Independently of the name we use to denote these monoids, the most important property they have, is that they are \emph{combinatorially finite}, that is to say, every element $\mathbf{a} \in S$ can be expressed only in finitely many ways as a sum $\mathbf{a} = \mathbf{a}_1 + \cdots + \mathbf{a}_q,$ with $\mathbf{a}_1,\ldots, \mathbf{a}_q \in S \setminus \{0\}$ (see \cite{BCMP}, or \cite{RGSU} for a wider class of monoids where this condition still holds true). Moreover, the binary relation on $S$ defined by $\mathbf{b} \prec_S \mathbf{a}$ if $\mathbf{a} - \mathbf{b} \in S$ is a well defined order on $S$ that satisfies the descending chain condition.

All monoids considered in this paper are finitely generated, commutative, cancellative and free of units, and thus we will omit these adjectives in what follows. Relevant examples of monoids fulfilling these conditions are \emph{affine semigroups}, that is monoids isomorphic to finitely generated submonoids of ${\mathbb N}^r$ with $r$ a positive integer (${\mathbb N}$ denotes here the set of nonnegative integers), and in particular, \emph{numerical semigroups} that are submonoids of the set of nonnegative integers with finite complement in ${\mathbb N}.$

We will write $S = \langle \mathbf{a}_1, \ldots, \mathbf{a}_r \rangle$ for the monoid generated by $\{\mathbf{a}_1, \ldots, \mathbf{a}_r\},$ that is to say, $S = \mathbf{a}_1 \mathbb{N} + \cdots + \mathbf{a}_r \mathbb{N}.$ In such a case, $\{\mathbf{a}_1, \ldots, \mathbf{a}_r\}$ will be said to be a \emph{system of generators} of $S.$ Moreover, if no proper subset of $\{\mathbf{a}_1, \ldots, \mathbf{a}_r\}$ generates $S,$ the set $\{\mathbf{a}_1, \ldots, \mathbf{a}_r\}$ is a \emph{minimal} system of generators of $S.$ In our context, every monoid has a unique minimal system of generators: define $S^*=S\setminus\{0\},$ then the minimal system of generators of $S$ is $S^*\setminus(S^*+S^*)$ (see \cite[Chapter 3]{RGS99}). In particular, if $S$ is the set of solutions of a system of linear Diophantine equations and/or inequalities, the minimal system of generators of $S$ coincides with the so called Hilbert Basis (see, e.g. \cite[Chapter 13]{Sturmfels95}).

Recall that if $S$ is a numerical semigroup minimally generated by $\{a_1< \cdots< a_r \} \subset \mathbb{N},$ the number $r$ is usually called \emph{embedding dimension} of $S,$ and the number $a_1$ is \emph{multiplicity}. It is easy to show (and well-known) that $\mathbf{a}_1 \geq r$ (see \cite[Proposition 2.10]{RGS09}). When $a_1 = r,$ $S$ is said to be of \emph{maximal embedding dimension}.

Given the minimal system of generators, $A = \{\mathbf{a}_1, \ldots, \mathbf{a}_r\},$ of a monoid $S,$ consider the monoid map $$\varphi_A : \mathbb{N}^r \longrightarrow S;\ \mathbf{u} = (u_1, \ldots, u_r) \longmapsto \sum_{i=1}^r u_i \mathbf{a}_i.$$ This map is sometimes known as the \emph{factorization homomorphism} associated to $S.$

Notice that each $\mathbf{u} = (u_1, \ldots, u_r) \in \varphi_A^{-1}(\mathbf{a})$ gives a \emph{factorization} of $\mathbf{a} \in S,$ say $\mathbf{a} = \sum_{i=1}^r u_i \mathbf{a}_i.$ Thus, $\# \varphi_A^{-1}(\mathbf{a})$ is the number of factorizations of $\mathbf{a} \in S.$ Observe that $\varphi_A^{-1}(\mathbf{a})$ is finite because of the combinatorial finiteness of $S$ (see also \cite[Lemma 9.1]{RGS99}).

Let $\sim_A$ be the kernel congruence of $\varphi_A,$ that is, $\mathbf{u}\sim_A \mathbf{v}$ if $\varphi_A(\mathbf{u})=\varphi_A(\mathbf{v})$ ($\sim_A$ is actually a congruence, an equivalence relation compatible with addition). It follows easily that $S$ is isomorphic to the monoid ${\mathbb N}^r/\sim_A.$

Given $\rho \subseteq {\mathbb N}^r\times {\mathbb N}^r,$ the congruence generated by $\rho$ is the least congruence containing $\rho.$ This congruence is the intersection of all congruences containing $\rho.$ If $\sim$ is the congruence generated by $\rho,$ then we say that $\rho$ is a \emph{system of generators}. R\'edei's theorem (see \cite{Redei}) precisely states that every congruence on ${\mathbb N}^r$ is finitely generated. A \emph{presentation} for $S$ is a system of generators of $\sim_A,$ and a \emph{minimal presentation} is a minimal system of generators of $\sim_A$ (in the sense that none of its proper subsets generates $\sim_A$). In our setting, all minimal presentations have the same cardinality (see for instance \cite{RGSU} or \cite{RGS99}). This is not the case for finitely generated monoids in general.

Next we briefly describe a procedure for finding all minimal presentations for $S$ as presented in \cite{RGSU} (in \cite[Chapter 9]{RGS99} this description is given in our context).

For $\mathbf{u} = (u_1, \ldots, u_r)$ and $\mathbf{v} = (v_1, \ldots, v_r) \in \mathbb{N}^r,$ we write $\mathbf{u} \cdot \mathbf{v}$ for $\sum_{i=1}^r u_i v_i$ (the dot product).

Given $\mathbf{a} \in S,$ we define the following binary relation on $\varphi_A^{-1}(\mathbf{a}).$ For $\mathbf{u}, \mathbf{u}' \in \varphi_A^{-1}(\mathbf{a}),$ $\mathbf{u} \mathcal{R} \mathbf{u}'$ if there exists a chain $\mathbf{u}_0, \ldots, \mathbf{u}_k \in \varphi_A^{-1}(\mathbf{a})$ such that
\begin{enumerate}[(a)]
\item $\mathbf{u}_0 = \mathbf{u},\ \mathbf{u}_k = \mathbf{u}',$
\item $\mathbf{u}_i \cdot \mathbf{u}_{i+1} \neq 0,\ i\in\{0, \ldots, k-1\}.$
\end{enumerate}
For every $\mathbf{a} \in S,$ define $\rho_\mathbf{a}$ in the following way.
\begin{itemize}
\item If $\varphi_A^{-1}(\mathbf{a})$ has one $\mathcal{R}-$class, then set $\rho_\mathbf{a}=\varnothing.$
\item Otherwise, let $\mathcal{R}_1,\ldots, \mathcal{R}_k$ be the different $\mathcal{R}-$classes of $\varphi_A^{-1}(\mathbf{a}).$ Choose $\mathbf{v}_i \in \mathcal{R}_i$ for all $i \in \{1,\ldots,k\}$ and set $\rho_\mathbf{a}$ to be any set of $k-1$ pairs of elements in $V=\{\mathbf{v}_1,\ldots,\mathbf{v}_k\}$ so that any two elements in $V$ are connected by a sequence of pairs in $\rho_\mathbf{a}$ (or their symmetrics). For instance, we can choose $\rho_\mathbf{a}=\{(\mathbf{v}_1, \mathbf{v}_2), \ldots, (\mathbf{v}_1, \mathbf{v}_k)\},$ or $\rho_\mathbf{a}=\{(\mathbf{v}_1, \mathbf{v}_2), (\mathbf{v}_2,\mathbf{v}_3),\ldots, (\mathbf{v}_{k-1}, \mathbf{v}_k)\}.$
\end{itemize}
Then $\rho=\bigcup_{\mathbf{a} \in S}\rho_\mathbf{a}$ is a minimal presentation of $S.$ Moreover, in this way one can
construct all minimal presentations for $S.$ Observe that there are finitely many elements $\mathbf{a}$ in $S$ for which $\varphi_A^{-1}(\mathbf{a})$ has more than one $\mathcal{R}-$class because $S$ is finitely presented.

\section{Betti elements}\label{Sect-Betti}

A minimal presentation of $S$ is as we have seen above a set of pairs of factorizations of some elements in $S,$ those having more than one ${\mathcal R}$-class. We say that $\mathbf{a}\in S$ is a \emph{Betti element} if $\varphi_A^{-1}(\mathbf{a})$ has more than one ${\mathcal R}$-class.

We will say the $\mathbf{a} \in S$ is \emph{Betti-minimal} if it is minimal among all the Betti elements in $S$ with respect to $\prec_S.$

Of course, Betti elements in $S$ are not necessarily Betti-minimal. Consider, for instance, $S = \langle 4, 6, 21 \rangle$ and $\mathbf{a} = 42.$

In the following, we will write $\mathrm{Betti}(S)$ and $\text{Betti-minimal}(S)$ for the sets of Betti elements and Betti minimal elements of the monoid $S,$ respectively.

\begin{lemma}\label{Lema 1}
Let $S = \langle \mathbf{a}_1, \ldots, \mathbf{a}_r \rangle.$ If $\mathbf{a} \not\in \mathrm{Betti}(S)$ and $\# \varphi_A^{-1}(\mathbf{a}) \geq 2,$ there exists $\mathbf{a}' \in \mathrm{Betti}(S)$ such that $\mathbf{a}' \prec_S \mathbf{a}.$
\end{lemma}
\begin{proof}
We will proceed by induction on $\# \varphi_A^{-1}(\mathbf{a}).$ If $\varphi_A^{-1}(\mathbf{a}) = \{\mathbf{u}, \mathbf{v}\}$ with $\mathbf{u} \cdot \mathbf{v} > 0,$ consider $\mathbf{a}' = \mathbf{a} - \sum_{i=1}^r \min(u_i,v_i) \mathbf{a}_i.$ Then, $\varphi_A^{-1}(\mathbf{a}') = \{\mathbf{u}', \mathbf{v}'\},$ with $u'_i = u_i - \min(u_i,v_i)$ and $v'_i = v_i - \min(u_i,v_i),\ i \in\{ 1, \ldots, r\},$ and $\mathbf{u}' \cdot \mathbf{v}' = 0.$ So, $\mathbf{a}' \prec \mathbf{a}$ is Betti. Assume now that the result is true for every $\mathbf{a}' \in S$ such that $2 \leq \# \varphi_A^{-1}(\mathbf{a}') < \# \varphi_A^{-1}(\mathbf{a}).$ Since $\mathbf{a}$ is not Betti, there exist $\mathbf{u}, \mathbf{v} \in \varphi_A^{-1}(\mathbf{a}), \mathbf{u} \neq \mathbf{v},$ such that $\mathbf{u} \cdot \mathbf{v} > 0.$ Consider $\mathbf{a}' = \mathbf{a} - \sum_{i=1}^r \min(u_i, v_i) \mathbf{a}_i.$ Then, we have that $2 \leq \# \varphi_A^{-1}(\mathbf{a}') \leq \# \varphi_A^{-1}(\mathbf{a}).$ If the second inequality is strict, we conclude by induction hypothesis. Otherwise, if $\mathbf{a}'$ is not Betti, we may repeat the previous argument to produce $\mathbf{a}'' \prec_S \mathbf{a}' \prec_S \mathbf{a}.$ The descending chain condition for $\prec_S$ guarantees that this process cannot continue indefinitely.
\end{proof}

\begin{remark}
Observe that the above lemma implies the existence of Betti elements in $S,$ when $S \not\cong \mathbb{N}^r,$ for any $r \geq 1.$ Otherwise, $\mathrm{Betti}(S) = \varnothing,$ because $\varphi_A$ is an isomorphism.
\end{remark}

Betti-minimal elements are characterized in the following result. As we will see later, they play an important role in the study of monoids with unique presentations.

\begin{proposition}\label{Prop1}
Let $S$ be a monoid. The element $\mathbf{a} \in \text{Betti-minimal}(S)$ if, and only, $\varphi_A^{-1}(\mathbf{a})$ has more than one $\mathcal{R}-$class and each $\mathcal{R}-$class is a singleton.
\end{proposition}
\begin{proof}
First, observe that $\varphi_A^{-1}(\mathbf{a})$ has more than one $\mathcal{R}-$class and each $\mathcal{R}-$class is a singleton if, and only if, $\# \varphi_A^{-1}(\mathbf{a}) \geq 2$ and $\mathbf{u} \cdot \mathbf{v} = 0,$ for every $\mathbf{u}, \mathbf{v} \in \varphi_A^{-1}(\mathbf{a}),\ \mathbf{u} \neq \mathbf{v}.$

If $\mathbf{a} \in \text{Betti-minimal}(S)$ and there exist $\mathbf{u}, \mathbf{v} \in \varphi_A^{-1}(\mathbf{a}),\ \mathbf{u} \neq \mathbf{v},$ such that $\mathbf{u} \cdot \mathbf{v} > 0,$ we consider $\mathbf{a}' = \mathbf{a} - \sum_{i=1}^r \min(u_i,v_i) \mathbf{a}_i.$ Since $\# \varphi_A^{-1}(\mathbf{a}') \geq 2,$ either $\mathbf{a}' \prec_S \mathbf{a}$ is Betti or, by Lemma \ref{Lema 1}, there exist $\mathbf{a}'' \in \mathrm{Betti}(S)$ such that $\mathbf{a}'' \prec_S \mathbf{a}' \prec_S \mathbf{a},$ contradicting, in both cases, the Betti-minimality of $\mathbf{a}.$ Conversely,
we suppose that $$\varphi_A^{-1}(\mathbf{a}) = \bigcup_{i=1}^{\# \varphi_A^{-1}(\mathbf{a})} \big\{ \mathbf{u}^{(i)} \big\},$$ with $\mathbf{u}^{(i)} \cdot \mathbf{u}^{(j)} = 0,\ i \neq j.$ In particular, $\mathbf{a} \in \mathrm{Betti}(S).$ If $\mathbf{a}' \prec_S \mathbf{a},$ then $\# \varphi_A^{-1}(\mathbf{a}') = 1,$ otherwise, we will find $i \neq j$ with $\mathbf{u}^{(i)} \cdot \mathbf{u}^{(j)} \neq 0.$ Thus we conclude that $\mathbf{a} \in \text{Betti-minimal}(S).$
\end{proof}

Observe that the notion of Betti-minimal is stronger than the notion of minimal multi-element given in \cite{Aoki}. Concretely, one has that $\mathbf{a} \in S$ is a minimal multi-element if, and only if, $\varphi_A^{-1}(\mathbf{a})$ has more than one $\mathcal{R}-$class and at least one of them is a singleton (see \cite[Definition 3.2]{Aoki}).

\section{Monoids having a unique minimal presentation}\label{Sect-UPS}

According to what we have recalled and defined so far, a monoid $S$ has a unique minimal presentation if and only if the set of factorizations of all its Betti elements have just two $\mathcal R$-classes, and each of them is a singleton. Moreover, if $\mathbf{a}$ is a Betti element of $S$ and $\varphi_A^{-1}(\mathbf{a})=\{\mathbf{u},\mathbf{v}\},$ then either the pair $(\mathbf{u},\mathbf{v})$ or $(\mathbf{v},\mathbf{u})$ is in any minimal presentation of $S.$ Hence we will say that $(\mathbf{u}, \mathbf{v}) \in \mathbb{N}^r \times \mathbb{N}^r$ is \emph{indispensable}, and that $\mathbf{a}$ has \emph{unique presentation}.

\begin{example}
The numerical semigroup $S = \langle 6,10,15 \rangle$ has no indispensable elements. If one uses the techniques explained in \cite{RGS09}, one can easily see that $\mathrm{Betti}(S)=\{30\},$ and that the factorizations of $30$ are $\{ ( 0, 0, 2 ), ( 0, 3, 0 ), (5,$ $0, 0 ) \}.$ One can also use the {\tt numericalsgps} {\tt GAP} package to perform this computation (see \cite{numericalsgps}).
\end{example}

Clearly, $S$ admits a unique minimal presentation if and only if either it is isomorphic to $\mathbb N^r$ for some positive integer $r$ (and thus the empty set is its unique minimal presentation) or every element in any of its minimal presentations is indispensable. If this is the case, we say that $S$ has a \emph{unique presentation}.

The following results are straightforward consequences of Proposition \ref{Prop1}.

\begin{corollary}\label{Cor mUP1}
Let $\mathbf{a} \in S.$ The following are equivalent.
\begin{enumerate}[(a)]
\item $\mathbf{a}$ has unique presentation.
\item $\mathbf{a} \in \mathrm{Betti}(S)$ and $\# \varphi_A^{-1}(\mathbf{a}) = 2.$
\item $\mathbf{a} \in \text{Betti-minimal}(S)$ and $\# \varphi_A^{-1}(\mathbf{a}) = 2.$
\end{enumerate}
\end{corollary}

\begin{corollary}\label{Cor mUP2}
A monoid $S$ is uniquely presented if, and only if, either $\mathrm{Betti}(S) = \varnothing$ or the number of Betti-minimal elements in $S$ equals the cardinality of a minimal presentation of $S.$ In particular all Betti elements of $S$ are Betti-minimal.
\end{corollary}

By using the close relationship between toric ideals and semigroups, one can obtain necessary and sufficient conditions for a semigroup to be uniquely presented from the results in \cite{Charalambous07, OjVi2, Takemura}.

\begin{example}\label{4621}
The above characterization does not hold if we remove the minimal condition. For instance, $S=\langle 4,6,21\rangle$ has a minimal presentation with cardinality 2, and $\mathrm{Betti}(S)=\{12,42\}$ (one can use the {\tt numericalsgps} package to compute this, \cite{numericalsgps}). However, 42 admits 5 different factorizations in $S.$
\end{example}

\begin{example}\label{Ex ED2}
Let $S\subset \mathbb Z^r$ be a monoid minimally generated by $A = \{\mathbf{a}_1, \mathbf{a}_2\}$ for some positive integer $r.$ If the rank of the group spanned by $S$ is one, there exist $u$ and $v \in \mathbb{N}$ such that $u \mathbf{a}_1 = v \mathbf{a}_2$. So, there is only one Betti element $\mathbf{a} = u \mathbf{a}_1 = v \mathbf{a}_2$ and $\varphi_A^{-1}(\mathbf{a}) = \big\{ (u,0), (0,v) \big\}$. Therefore, $S$ is uniquely presented. In particular, embedding dimension $2$ numerical semigroups are uniquely presented (the group generated by any numerical semigroup is $\mathbb Z$).
\end{example}

\section{Gluings}\label{Sect Glue}

We first fix the notation of this section. Let $S$ be an affine semigroup generated by $A = \{\mathbf{a}_1, \ldots, \mathbf{a}_r\} \subseteq \mathbb{Z}^n.$ Let $A_1$ and $A_2$ be two proper subsets of $A$ such that $A = A_1 \cup A_2$ and $A_1 \cap A_2 = \varnothing.$ Let $S_1$ and $S_2$ be the affine semigroups generated by $A_1$ and $A_2,$ respectively.

Set $r_1$ and $r_2$ to be the cardinality of $A_1$ and $A_2,$ respectively. After rearranging the elements of $A$ if necessary, we may assume that $A_1 = \{\mathbf{a}_1, \ldots, \mathbf{a}_{r_1}\}$ and $A_2 = \{\mathbf{a}_{r_1 + 1}, \ldots, \mathbf{a}_r\}.$

Since $\mathbb{N}^r = \mathbb{N}^{r_1} \oplus \mathbb{N}^{r_2},$ elements in $\mathbb{N}^{r_1}$ and $\mathbb{N}^{r_2}$ may be regarded as elements in $\mathbb{N}^r$ of the form $(-, 0)$ and $(0, -),$ respectively. With this in mind, subsets of $\mathbb{N}^{r_i}$ will be considered as subsets of $\mathbb{N}^r,\ i \in\{ 1,2\}.$ And the elements of $\sim_{A_1}$ and $\sim_{A_2}$ are viewed inside $\sim_A.$

The monoid $S$ is said to be the \emph{gluing} of $S_1$ and $S_2$ if $G(S_1) \cap G(S_2) = \mathbf{d} \mathbb{Z},$ with $\mathbf{d} \in S_1 \cap S_2 \setminus \{0\},$ where $G(-)$ means the group generated by $-.$

According to \cite[Theorem 1.4]{Ros97}, $S$ admits a presentation of the form $\rho_1\cup\rho_2\cup\{((\mathbf{u},0),(0,\mathbf{v}))\},$ where $\rho_1$ and $\rho_2$ are presentations of $S_1$ and $S_2,$ respectively, and $\mathbf{u}\in \varphi_{A_1}^{-1}(\mathbf{d})$ and $\mathbf{b}\in\varphi_{A_2}^{-1}(\mathbf{d}).$ We next explore which are the conditions we must impose on $S_1,$ $S_2$ and $\mathbf{d}$ in order to ensure that $S$ has a unique minimal presentation. We start by describing the Betti elements of $S,$ and for this we need a lemma describing how are the factorizations of $\mathbf{d}.$

\begin{lemma}\label{Lema Glue1}
Let $S$ be the gluing of $S_1$ and $S_2$ with $G(S_1) \cap G(S_2) = \mathbf{d} \mathbb{Z}.$ Every factorization of $\mathbf{d}$ in $S$ is either a factorization of $\mathbf{d}$ in $S_1$ or a factorization of $\mathbf{d}$ in $S_2.$ In particular $\mathbf{d} \in \mathrm{Betti}(S).$
\end{lemma}
\begin{proof}
By definition $\mathbf{d} \in S_1 \cap S_2 \setminus \{0\},$ so, there exist $\mathbf{u} \in \mathbb{N}^{r_1}$ and $\mathbf{v} \in \mathbb{N}^{r_2}$ such that $\mathbf{d} = \sum_{i = 1}^{r_1} u_i \mathbf{a}_i = \sum_{i = r_1 + 1}^{r} v_i \mathbf{a}_i.$ If $\mathbf{d} = \sum_{i = 1}^r w_i \mathbf{a}_i = \sum_{i = 1}^{r_1} w_i \mathbf{a}_i + \sum_{i = r_1 + 1}^{r} w_i \mathbf{a}_i,$ then $$\mathbf{d} - \sum_{i = 1}^{r_1} w_i \mathbf{a}_i = \sum_{i = 1}^{r_1} u_i \mathbf{a}_i - \sum_{i = 1}^{r_1} w_i \mathbf{a}_i = \sum_{i = r_1 + 1}^{r} w_i \mathbf{a}_i \in G(S_1) \cap G(S_2),$$ that is to say, $\mathbf{d} - \sum_{i = 1}^{r_1} w_i \mathbf{a}_i = z \mathbf{d}.$ Therefore, either $z = 1$ and then $w_i = 0,\ i\in \{ 1, \ldots, r_1\},$ or $z = 0$ and then $w_i = 0,\ i \in\{ r_1 + 1, \ldots, r\},$ as claimed.

Moreover, we have that $\varphi_A^{-1}(\mathbf{d}) = \varphi_{A_1}^{-1}(\mathbf{d}) \cup \varphi_{A_2}^{-1}(\mathbf{d})$ with $(\mathbf{u}, 0) \cdot (0, \mathbf{v}) = 0$ for every $\mathbf{u} \in \varphi_{A_1}^{-1}(\mathbf{d})$ and $\mathbf{v} \in \varphi_{A_2}^{-1}(\mathbf{d}),$ which means that $\varphi_A^{-1}(\mathbf{d})$ has at least two $\mathcal R$-classes. Hence $\mathbf{d} \in \mathrm{Betti}(S).$
\end{proof}

\begin{theorem}\label{Th Glue1}
Let $S$ be the gluing of $S_1$ and $S_2$, and $G(S_1) \cap G(S_2) = \mathbf{d} \mathbb{Z}.$ Then, $$\mathrm{Betti}(S) = \mathrm{Betti}(S_1) \cup \mathrm{Betti}(S_2) \cup \{\mathbf{d}\}.$$
\end{theorem}

\begin{proof}
By Theorem 1.4 in \cite{Ros97}, $S$ admits a presentation of the form $\rho = \rho_1 \cup \rho_2 \cup \{\big( (\mathbf{u},0), (0, \mathbf{v}) \big) \},$ where $\rho_1$ and $\rho_2$ are sets of generators for $\sim_{A_1}$ and $\sim_{A_2},$ respectively, and $\varphi_{A_1}(\mathbf{u}) = \varphi_{A_2}(\mathbf{v}) = \mathbf{d}.$ Moreover, since every system of generators of $\sim_A$ can be refined to a minimal system of generators (see \cite[Chapter, 9]{RGS99}), from the shape of $\rho,$ we deduce that the Betti elements of $S$ are either a Betti element of $S_1,$ a Betti element of $S_2$ or $\mathbf{d}$ itself, that is to say, $\mathrm{Betti}(S) \subseteq \mathrm{Betti}(S_1) \cup \mathrm{Betti}(S_2) \cup \{\mathbf{d}\}.$

Recall that, by Lemma \ref{Lema Glue1}, $\mathbf{d} \in \mathrm{Betti}(S).$ Therefore, to demonstrate the inclusion $\mathrm{Betti}(S) \supseteq \mathrm{Betti}(S_1) \cup \mathrm{Betti}(S_2) \cup \{\mathbf{d}\},$ it suffices to prove $\mathrm{Betti}(S_1) \cup \mathrm{Betti}(S_2) \subseteq \mathrm{Betti}(S).$ Suppose, in order to produce a contradiction, that there is $\mathbf{b} \in \mathrm{Betti}(S_1)\setminus \mathrm{Betti}(S)$ (the case where $\mathbf{b} \in \mathrm{Betti}(S_2)\setminus \mathrm{Betti}(S)$ is argued similarly).

Since $\mathbf{b} \in \mathrm{Betti}(S_1),$ there exist two $\mathcal R$-classes in $\varphi^{-1}_{A_1}(\mathbf{b}),$ say $\mathcal{C}_1$ and $\mathcal{C}_2.$ And as $\mathbf{b} \not\in \mathrm{Betti}(S),$ $\varphi^{-1}_A(\mathbf{b})$ has only one $\mathcal R$-class. Hence there exist
\begin{itemize}
\item $\mathbf{w} \in \mathcal C_1$ and $\bar{\mathbf{w}} \in \varphi^{-1}_A(\mathbf{b})$
such that $\bar{\mathbf{w}} \cdot (\mathbf{w},0) \neq 0$ and $\mathbf{b} = \sum_{i = 1}^{r_1} \bar{w}_i \mathbf{a}_i + \sum_{i = r_1 + 1}^{r} \bar{w}_i \mathbf{a},$ where $\bar{w}_i,\ 1 \leq i \leq r,$ are the coordinates of $\bar{\mathbf{w}}$ and $\bar{w}_i \neq 0$ for some $r_1 + 1 \leq i \leq r.$
\item $\mathbf{w}' \in \mathcal C_2$ and $\bar{\mathbf{w}}' \in \varphi^{-1}_A(\mathbf{b})$
such that $\bar{\mathbf{w}}' \cdot (\mathbf{w}',0) \neq 0$ and $\mathbf{b} = \sum_{i = 1}^{r_1} \bar{w}'_i \mathbf{a}_i + \sum_{i = r_1 + 1}^{r} \bar{w}'_i \mathbf{a}_i,$ where $\bar{w}'_i,\ 1 \leq i \leq r,$ are the coordinates of $\bar{\mathbf{w}}'$ and $\bar{w}'_i \neq 0$ for some $r_1 + 1 \leq i \leq r.$
\end{itemize}
Since $0 \neq \mathbf{b} - \sum_{i = 1}^{r_1} \bar{w}_i \mathbf{a}_i = \sum_{i = r_1 + 1}^{r} \bar{w}_i \mathbf{a}_i \in G(S_1) \cap G(S_2) = \mathbf{d} \mathbb{Z},$ we have that $\mathbf{b} = \sum_{i = 1}^{r_1} \bar{w}_i \mathbf{a}_i + \sum_{i = 1}^{r_1} z u_i \mathbf{a}_i = \sum_{i = 1}^{r_1} (\bar{w}_i + z u_i) \mathbf{a}_i,$ for some $z > 0.$ Analogously, $\mathbf{b} = \sum_{i = 1}^{r_1} (\bar{w}'_i + z' u_i) \mathbf{a}_i,$ for some $z' > 0.$

Let $\tilde{\mathbf{w}}$ and $\tilde{\mathbf{w}}' \in \varphi^{-1}_{A_1}(\mathbf{b})$ be the corresponding vectors of coordinates $\bar{w}_i + z u_i, 1 \leq i \leq r_1$ and $\bar{w}'_i + z' u_i, 1 \leq i \leq r_1,$ respectively. This yields a contradiction, since $\mathbf{w}$ and $\mathbf{w}'$ are not $\mathcal R$-related, however $\mathbf{w} \cdot \tilde{\mathbf{w}} \neq 0,$ $\tilde{\mathbf{w}} \cdot \tilde{\mathbf{w}}'\neq 0$ and $\tilde{\mathbf{w}}'\cdot \mathbf{w}' \neq 0.$
\end{proof}

Observe that $\varphi_A^{-1}(\mathbf{d}) \supseteq \{ (\mathbf{u},0), (0, \mathbf{v}) \},$ with $\varphi_{A_1}(\mathbf{u}) = \varphi_{A_2}(\mathbf{v}) = \mathbf{d},$ and that the equality holds if, and only if, $\mathbf{d}$ has unique presentation as element of $S.$

\begin{corollary}\label{Cor Glue1}
Let $S$ be the gluing of $S_1$ and $S_2$ and $G(S_1) \cap G(S_2) = \mathbf{d} \mathbb{Z}.$ Then $\mathbf{d} \in S$ has unique presentation if, and only if, $\mathbf{d} - \mathbf{a} \not\in S$ for every $\mathbf{a} \in \mathrm{Betti}(S_1)\cup \mathrm{Betti}(S_2).$
\end{corollary}
\begin{proof}
If $\mathbf{d}$ has unique presentation then, by Corollary \ref{Cor mUP1}, $\mathbf{d}$
belongs to $\text{Betti-minimal}(S)$.
So, $\mathbf{d} - \mathbf{a} \not\in S$ for every $\mathbf{a} \in \mathrm{Betti}(S) \setminus \{\mathbf{d}\}.$ Since $\mathbf{d} \not\in \mathrm{Betti}(S_1) \cup \mathrm{Betti}(S_2)$ (because $\mathbf{d}$ has unique factorization in $S_i, i \in\{ 1,2\}$), by Theorem \ref{Th Glue1}, $\mathrm{Betti}(S) \setminus \{\mathbf{d}\} = \mathrm{Betti}(S_1) \cup \mathrm{Betti}(S_2).$ Thus, we conclude that $\mathbf{d} - \mathbf{a} \not\in S$ for every $\mathbf{a} \in \mathrm{Betti}(S_1) \cup \mathrm{Betti}(S_2).$

Conversely, in view of Lemma \ref{Lema 1}, we deduce that $\mathbf{d}$ admits a unique factorization in $S_i,\ i\in\{1,2\},$ that is to say, $\varphi_{A_1}^{-1}(\mathbf{d}) = \{ \mathbf{u} \}$ and $\varphi_{A_2}^{-1}(\mathbf{d}) = \{ \mathbf{v} \}.$ Since by Lemma \ref{Lema Glue1} we have that $\mathbf{d}$ is a Betti element, we conclude that $\varphi^{-1}_A(\mathbf{d}) = \{(\mathbf{u},0), (0, \mathbf{v})\}.$
\end{proof}

\begin{theorem}\label{Th Glue2}
Let $S$ be the gluing of $S_1$ and $S_2$,
and $G(S_1) \cap G(S_2) = \mathbf{d} \mathbb{Z}.$ Then, $S$ is uniquely presented if, and only if,
\begin{enumerate}[(a)]
\item $S_1$ and $S_2$ are uniquely presented,
\item $\pm (\mathbf{d} - \mathbf{a}) \not\in S,$ for every $\mathbf{a} \in \mathrm{Betti}(S_1) \cup \mathrm{Betti}(S_2),$
\end{enumerate}
\end{theorem}
\begin{proof}
By Theorem \ref{Th Glue1}, $\mathrm{Betti}(S) = \mathrm{Betti}(S_1) \cup \mathrm{Betti}(S_2) \cup \{\mathbf{d}\}.$ So, if $S$ is uniquely presented, then every $\mathbf{a} \in \mathrm{Betti}(S_1) \cup \mathrm{Betti}(S_2) \cup \{\mathbf{d}\}$ has unique presentation. Thus, $S_1$ and $S_2$ are uniquely presented and, by Corollary \ref{Cor Glue1}, $\mathbf{d} - \mathbf{a} \not\in S,$ for every $\mathbf{a} \in \mathrm{Betti}(S_1) \cup \mathrm{Betti}(S_2).$ Finally, since, by Corollary \ref{Cor mUP1}, every $\mathbf{a} \in \mathrm{Betti}(S)$ is Betti-minimal, we conclude that $\mathbf{a} - \mathbf{d} \not\in S,$ for every $\mathbf{a} \in \mathrm{Betti}(S_1) \cup \mathrm{Betti}(S_2)$ (note that $\mathbf{d} - \mathbf{m} \not\in S$ implies $\mathbf{d} \neq \mathbf{m},$ for every $\mathbf{m} \in \mathrm{Betti}(S_1) \cup \mathrm{Betti}(S_2)$).

Conversely, suppose that Conditions (a) and (b) hold. In particular, every $\mathbf{a} \in \mathrm{Betti}(S_i)$ has only two factorizations as element of $S_i,\ i\in\{1,2\}$ and, by Corollary \ref{Cor Glue1}, $\mathbf{d}$ has only two factorizations in $S,$ say $\mathbf{d} = \sum_{i = 1}^{r_1} u_i \mathbf{a}_i = \sum_{i = r_1 + 1}^r v_i \mathbf{a}_i.$ So, if $\mathbf{a} \in \mathrm{Betti}(S)$ has more than two factorizations in $S,$ then $\mathbf{d} \neq \mathbf{a} \in \mathrm{Betti}(S_1) \cup \mathrm{Betti}(S_2).$ If $\mathbf{a} \in \mathrm{Betti}(S_1),$ then $\mathbf{a} = \sum_{i = 1}^{r_1} w_i \mathbf{a}_i + \sum_{i = r_1 + 1}^r c_i \mathbf{a}_i,$ with $w_i \neq 0,$ for some $r_1 + 1 \leq i \leq r.$ Thus, $\mathbf{a} - \sum_{i = 1}^{r_1} w_i \mathbf{a}_i = \sum_{i = r_1 + 1}^r w_i \mathbf{a}_i \in G(S_1) \cap G(S_2) = \mathbf{d} \mathbb{Z}$ and thus, $\mathbf{a} - \mathbf{d} \in S$ which is impossible by hypothesis.
\end{proof}

The affine semigroup in the following example is borrowed from \cite{RGS99a} where the authors use it to illustrate their algorithm for checking freeness of simplicial semigroups. We use $\mathbf{e}_i\in \mathbb N^r$ to denote the $i$th row of the identity $r\times r$ matrix.

\begin{example}
Let us see that $S = \langle (2,0), (0,3), (2,1), (1,2) \rangle$ is uniquely presented. On the one hand, by taking $A_1 = \{(2,0), (0,3),$ $(2,1)\}, A_2 = \{(1,2) \}, S_1 = \langle A_1 \rangle$ and $S_2 = \langle A_2 \rangle,$ we have that $G(S_1) \cap G(S_2) = 2(1,2) \mathbb{Z}.$ On the other hand, by taking $A_{11} = \{(2,0), (0,3)\}, A_{12} = \{(2,1)\}, S_{11} = \langle A_{11} \rangle$ and $S_{12} = \langle A_{12} \rangle,$ we have that $G(S_{11}) \cap G(S_{12}) = 3(2,1) \mathbb{Z}.$ Since $S_{11} \cong \mathbb{N}^2$ and $S_{12} \cong \mathbb{N}$ are uniquely presented (because, their corresponding presentations are the empty set) and Condition (b) in Theorem \ref{Th Glue2} is trivially satisfied, we may assure that $S_1$ is uniquely presented by $\{(3 \mathbf{e}_3, 3 \mathbf{e}_1 + \mathbf{e}_2 )\}.$ Finally, since $S_1$ and $S_2 \cong \mathbb{N}$ are uniquely presented and $2(1,2) - 3(2,1) \not\in S$ we conclude that $S$ is uniquely presented by $\{(3 \mathbf{e}_3 ,3 \mathbf{e}_1 + \mathbf{e}_2 ),\ (2 \mathbf{e}_4, \mathbf{e}_2 + \mathbf{e}_3)\}.$
\end{example}

\begin{example}\label{Ejem Glue1}
In this example we construct an infinite sequence of uniquely presented numerical semigroups. Let us start with $S_1=\langle 2,3\rangle,$ and given $S_i$ minimally generated by $\{ a_1,\ldots,a_{i+1}\},\ i \geq 2,$ set $S_{i+1}=\langle 2a_1,a_1+a_2, 2a_2,\ldots, 2a_{i+1}\rangle.$ We prove by induction on $i$ that $S_{i+1}$ is uniquely presented by
\begin{align*}
\rho_{i+1} = \big\{ (2\mathbf{e}_2, \mathbf{e}_1+\mathbf{e}_3),(2\mathbf{e}_3,\mathbf{e}_1+\mathbf{e}_4),\ldots, (2\mathbf{e}_i,\mathbf{e}_1+\mathbf{e}_i),(2\mathbf{e}_{i+1},3\mathbf{e}_1) \big\}.
\end{align*}

For $i=1$ the result follows easily. Assume that $i \geq 2$ and that the result holds for $S_i$ and let us show it for $S_{i+1}.$ Observe that $S_{i+1}$ is the gluing of $\langle 2a_1,\ldots,2a_{i+1}\rangle=2S_i$ and $\langle a_1+a_2\rangle,$ with $d=2a_1+2a_2,$ and consequently $S_{i+1}$ is minimally generated by $\{ 2a_1,a_1+a_2, 2a_2,\ldots, 2a_{i+1}\}$ (see Lemma 9.8 in \cite{RGS09} with $\lambda=2$ and $\mu=a_1+a_2$). Notice that
${\rm Betti}(\langle a_1+a_2\rangle)=\varnothing$ and, by induction hypothesis, ${\rm Betti}(2S_i)=2{\rm Betti}(S_i)=\{2(2a_2),\ldots,2(2a_{i+1})\}.$ Thus, by Theorem \ref{Th Glue1}, $$\mathrm{Betti}(S_{i+1}) = \{d\} \cup \mathrm{Betti}(2 S_i) = \{ 2a_1+2a_2,2(2a_2),\ldots,2(2a_{i+1})\}.$$ Now, a direct computation shows that $\rho_{i+1}$ is a minimal presentation of $S_{i+1}.$

In view of Theorem \ref{Th Glue2}, to prove the uniqueness of the presentation, it suffices to check that for $b=2(2a_j)-(2a_1+2a_2),$ neither $b$ nor $-b$ belong to $S_{i+1}.$ Observe that $-b < 0,$ since $j \geq 2,$ and thus it is not in $S_{i+1}.$ Besides, if $j \neq i,$ then $2(2a_j)-(2a_1+2a_2) = 2a_1 + 2a_{j+1} - 2a_1 - 2 a_2 = 2 a_{j+1} - 2a_2.$ This element cannot be in $S_{i+1}$ because $2a_{j+1}$ is one of its minimal generators. For $j=i,$ we get $2(2a_{i+1})-(2a_1+2a_2)=2(3a_1)-2a_1-2a_2=2(2a_1)-2a_2.$ If this integer belongs to $S_{i+1},$ then by the minimality of $2a_2,$ there exists $a \in S_{i+1}\setminus\{0\}$ such that $2(2a_1)=2a_2 + a.$ But then $a \geq 2a_1,$ and as $2a_2>2a_1,$ we get a contradiction.

For every positive integer $i,$ the numerical semigroup $S_{i+1}$ is a free numerical semigroup in the sense of \cite{Bertin}, and thus it is a complete intersection (numerical semigroup with minimal presentations with the least possible cardinality: the embedding dimension minus one). Some authors call these semigroups telescopic. Not all free numerical semigroups have unique minimal presentation; $\langle 4,6,21\rangle$ illustrates this fact (see Example \ref{4621}).
\end{example}

\section{Uniquely presented numerical semigroups}\label{Sect NS}

We would like to mention that there are ``few'' numerical semigroups having unique minimal presentation. The following sequences have been computed with the {\tt numericalsgps} {\tt GAP} package (\cite{numericalsgps}). The first contains in the $i$th position the number of numerical semigroups with \emph{Frobenius number} $i\in\{1,\ldots, 20\}$ (meaning that $i$ is the largest integer not in the semigroup), and the second contains those with the same condition having a unique minimal presentation.
$$( 1, 1, 2, 2, 5, 4, 11, 10, 21, 22, 51, 40, 106, 103, 200, 205, 465, 405, 961, 900 ),$$
$$ (1, 1, 1, 1, 3, 1, 5, 2, 5, 4, 8, 2, 12, 8, 6, 9, 17, 8, 20, 12).$$

Next we explore three big families of numerical semigroups, and determine its elements having unique minimal presentations.

\subsection{Numerical semigroups generated by intervals}

Let $a$ and $x$ be two positive integers, and let $S=\langle a, a+1,\ldots, a+x\rangle.$ Since $\mathbb{N}$ is uniquely presented, we may assume that $2\leq a.$ In order that $\{a,\ldots,a+x\}$ becomes a minimal system of generators for $S,$ we suppose that $x<a.$

\begin{theorem}
$S=\langle a,a+1,\ldots,a+x\rangle$ ($x<a$) is uniquely presented if, and only if, either $a=1,$ (that is, $S = \mathbb{N}$) or $x=1,$ or $x=2,$ or $x=3$ and $(a-1)\bmod x\neq 0.$
\end{theorem}

\begin{proof}
The Betti elements in $S$ are fully described in \cite[Theorem 8]{intervalos}, so we will make an extensive use of this result. If $x \geq 4,\ m = 2(a+2)$ is a Betti element and $\# \varphi_A^{-1}(m) = 3.$ Thus for $x \geq 4,$ $S$ is not uniquely presented. Hence we focus on $x \in \{1,2,3\}.$ For simplicity in the forthcoming notation, let $q$ and $r$ be the quotient and the remainder in the division of $a-1$ by $x,$ that is to say, $a = xq + r + 1$ with $0 \leq r \leq x-1.$ Notice that $x < a$ implies $q \geq 1.$

For $x=1,$ we get an embedding dimension two numerical semigroup which is uniquely presented (see Example \ref{Ex ED2}).

For $x=2,$
\[ \mathrm{Betti}(S) = \left\{
\begin{array}{ll}
\{2(a+1), qa + 2(q-1) + 1, qa + 2(q-1) + 2\}, & r=0,\\
\{2(a+1), qa + 2(q-1) + 2\}, & r=1.
\end{array}\right.
\]
Since the cardinality of a minimal presentation of $S$ is $3 - r$ (\cite[Theorem 8]{intervalos}), by Corollary \ref{Cor mUP2}, we only must check whether or not they are incomparable with respect to $\prec_S.$ If $r = 0,$ clearly $qa + 2(q-1) + 1$ and $qa + 2(q-1) + 2$ are incomparable, since $1 \not\in S.$ Besides, $qa + 2(q-1) + 1 - 2(a+1) = (q-1)a + 2q - 1\not\in S$ in view of \cite[Lemma 1]{intervalos} ($2q-1>2(q-1)$), and the same argument applies to $qa + 2(q-1) + 2 - 2(a+1) = (q-1)a+2q.$ If $r=1,$ $qa+2(q-1)+2 - 2(a+1) = (q-2)a + 2(q-1) \not\in S$ (use again \cite[Lemma 1]{intervalos}), we also obtain a (complete intersection) uniquely presented numerical semigroup. Hence every numerical semigroup of the form $\langle a, a+1, a+2 \rangle,$ with $a \geq 3,$ is uniquely presented.

Assume that $x=3$ (and thus $a\geq 4$).
\begin{itemize}
\item[$r=0.$] In this setting, both $(q+1)(a+3)$ and $2(a+1)$ are Betti elements. However, $(q+1)(a+3)-2(a+1)=(q-1)a+q 3+1=(q-1)a+(a-1)+1=q a\in S.$ Hence $(q+1)(a+3) \not\in \text{Betti-minimal}(S)$ and so, by Corollary \ref{Cor mUP2}, it is not uniquely presented.
\item[$r\neq 0.$] In this case,
  \[\mathrm{Betti}(S) = \left\{ \begin{array}{lll}
\begin{array}{l} \big\{2(a+1),\ (a+1)+(a+2),2(a+2),\\ \phantom{\{} qa + 3(q-1) + 2,\ qa + 3(q-1) + 3 \big\} \end{array} & \text{if} & r=1,\\ \begin{array}{l} \big\{2(a+1),\ (a+1)+(a+2),\\ \phantom{\{} 2(a+2),\ qa + 3(q-1) + 3 \big\}\end{array} & \text{if} & r=2.
  \end{array}\right.\]
   Since the cardinality of a minimal presentation of $S$ is $6 - r$ (\cite[Theorem 8]{intervalos}), by Corollary \ref{Cor mUP2}, we only must check whether or not they are incomparable with respect to $\prec_S.$ Observe that $q a+(q-1)3+j-2a-i=(q-2)a+(q-1)3+j-i\not \in S$ if and only if $q+j+1>i$ (\cite[Lemma 1]{intervalos}). As in our case $i\in \{2,3,4\},$ $j\in\{2,3\}$ and $q\geq 1,$ we obtain that these elements are incomparable. Thus, $S$ is uniquely presented.
\end{itemize}
\end{proof}

\subsection{Embedding dimension three numerical semigroups}

As we have pointed out above, the Frobenius number of a numerical semigroup is the largest integer not belonging to it. A numerical semigroup $S$ with Frobenius number $f$ is \emph{symmetric} if for every $x\in \mathbb Z\setminus S,$ $f-x\in S.$ For embedding dimension three numerical semigroups it is well-known that the concept of symmetric and complete intersection numerical semigroups coincide (and also free, see for instance \cite[Chapter 10]{RGS09} or \cite{Herzog70}). Non-symmetric numerical semigroups with embedded dimension three are uniquely presented (\cite{Herzog70}). Thus, we will center our attention in the symmetric case, which is the free case, and as Delorme proved in \cite{Delorme}, these semigroups are the gluing of an embedding dimension two numerical semigroup and $\mathbb N$ (see \cite{Ros97} for a proof using the concept of gluing).
So every symmetric numerical semigroup with embedding dimension three can be described as follows.

\begin{proposition}{\cite[Theorem 10.6]{RGS09}}\label{libres-de-tres}
Let $m_1$ and $m_2$ two relatively prime integers greater than one. Let $a, b$ and $c$ be nonnegative integers with $a \geq 2,\ b + c \geq 2$ and $\gcd(a, b m_1 + c m_2) = 1.$ Then $S = \langle a m_1, a m_2, b m_1 + c m_2 \rangle$ is a symmetric numerical semigroup with embedding dimension three. Moreover, every embedding dimension three symmetric numerical semigroup is of this form.
\end{proposition}

Now, our main result is just a particularization of what we have already seen in Section \ref{Sect Glue}.

\begin{theorem}\label{Th SS5.2}
With the same notation as in Proposition \ref{libres-de-tres}, $S$ is a symmetric numerical semigroup uniquely presented with embedding dimension three, if and only if, $0 < b < m_2$ and $0 < c < m_1.$
\end{theorem}

For the proof of this result, we will need the following lemma.

\begin{lemma}\label{tres}
Let $m_1$ and $m_2$ two relatively prime integers greater than one. Then, $m_1 m_2 = \alpha m_1 + \beta m_2,$ for some $\alpha \geq 0$ and $\beta \geq 0,$ if, and only if, $\alpha = m_2$ and $\beta = 0,$ or $\alpha = 0$ and $\beta = m_1.$
\end{lemma}
\begin{proof}
$m_1 m_2 = \alpha m_1 + \beta m_2,$ for some $\alpha \geq 0$ and $\beta \geq 0,$ if, and only if, $(m_2 - \alpha) m_1 = \beta m_2,$ for some $\alpha \geq 0$ and $\beta \geq 0.$ Since $\gcd(m_1, m_2) = 1,$ it follows that $(m_2 - \alpha) m_1 = \beta m_2,$ for some $\alpha \geq 0$ and $\beta \geq 0,$ if, and only if, $m_2-\alpha = \gamma m_2$ and $\beta = \gamma m_1$ for some $\gamma \geq 0,$ if, and only if, $\alpha = (1-\gamma) m_2$ and $\beta = \gamma m_1,$ for some $0 \leq \gamma \leq 1,$ if, and only if, $\alpha = m_2$ and $\beta = 0$ or $\alpha = 0$ and $\beta = m_1.$
\end{proof}

\medskip
\noindent\emph{Proof of Theorem \ref{Th SS5.2}.}
Since $S$ is the gluing of $S_1 = \langle a m_1, a m_2 \rangle$ and $S_2 = \langle b m_1 + c m_2 \rangle,$ with $d = a (b m_1 + c m_2),\ \mathrm{Betti}(S_1) = a m_1 m_2$ and $\mathrm{Betti}(S_2) = \varnothing,$ by Theorem \ref{Th Glue1}, $\mathrm{Betti}(S) = \{a m_1 m_2, a (b m_1 + c m_2)\}.$ Thus, by Theorem \ref{Th Glue2}, $S$ is uniquely presented if, and only if, $\pm(a m_1 m_2 - a (b m_1 + c m_2)) \not\in S.$

By direct computation, one can check that $a (b m_1 + c m_2) - a m_1 m_2 \in S$ if, and only if, $b \geq m_2$ or $c \geq m_1.$ Besides, $a m_1 m_2 - a (b m_1 + c m_2) \in S$ if, and only, if $m_1 m_2 = ((\alpha_3 + 1) b + \alpha_1) m_1 + ((\alpha_3 + c) c + \alpha_1) m_2,$ for some $\alpha_i \geq 0,\ i \{ 1,2,3\}.$ In view of Lemma \ref{tres}, this is equivalent to
$((\alpha_3 + 1) b + \alpha_1) = 0$ and $((\alpha_3 + c) c + \alpha_1) = m_1$ or $((\alpha_3 + 1) b + \alpha_1) = m_2$ and $((\alpha_3 + c) c + \alpha_1) = 0,$ for some $\alpha_i \geq 0,\ i \in\{ 1,2,3\}.$ And this holds if, and only if, $b = 0$ and $c \leq m_1$ or $b \leq m_2$ and $c = 0.$

Therefore, $\pm(a m_1 m_2 - a (b m_1 + c m_2)) \not\in S,$ if, and only if, $0 < b < m_2$ and $0 < c < m_1.$\mbox{}\hfill\qed

\subsection{Maximal embedding dimension numerical semigroups}

\begin{theorem}
A numerical semigroup $S$ minimally generated by $a_1 < a_2 < \cdots < a_r$ with $a_1 = r$ is uniquely presented if, only if, $r = 3.$
\end{theorem}

\begin{proof}
For $r=3,$ we obtain numerical semigroups of the form $\langle 3, a, b\rangle,$ with $a$ and $b$ not multiples of $3$ and thus coprime with 3. It follows easily that these semigroups have not the shape given in Theorem \ref{libres-de-tres}, and thus are not symmetric. Consequently, they are uniquely presented.

We now prove that that if $a_1 = r \geq 4,$ $S$ cannot be uniquely presented. According to \cite{Ros96b}, $\mathrm{Betti}(S)=\{a_i+a_j~|~ i,j\in\{2,\ldots,r\}\}.$ All the elements in $\{0,a_2,\ldots,a_r\}$ belong to different classes modulo $a_1,$ and there are precisely $a_1$ of them. Thus $2a_r$ can be uniquely be written as $b a_1+a_i$ for some $i\in \{2,\ldots,r-1\}$ and $b$ a positive integer.

Let $f$ be the Frobenius number of $S.$ It is well-known that $f=a_r-a_1$ in this setting (see for instance \cite{RGS09}). Since $2 a_r - a_i =a_r+(a_r-a_i)> a_r - a_1 = f,$ for all $i,$ it follows that $2 a_r - a_i \in S.$ Hence $2 a_r = a_i + m_i,\ m_i \in S$ for every $i\in \{1,\ldots,r\}.$ Take $i\not=k.$ Then $2a_r$ admits at least three expressions: $2a_r,$ $ba_1+a_k$ and $a_i+m.$ By Corollary \ref{Cor mUP1}, $S$ cannot have a unique minimal presentation.
\end{proof}

\noindent
\textbf{Acknowledgments}
The authors gratefully thanks Anargyros Katsabekis for a careful reading and correcting some misprints and to the referee for his or her valuable remarks.

Part of this work was done during a visit of the first author to the University of Extremadura financed by the Plan Propio 2009 of the University of Extremadura.

\end{document}